\documentclass[12pt]{amsart}
\usepackage{amsmath}
\usepackage{amstext}
\usepackage{amsfonts}
\usepackage{amssymb}
\usepackage{amsthm}

\theoremstyle{plain}
\newtheorem{thm}{Theorem}[section]
\newtheorem{cor}[thm]{Corollary}

\newtheorem{lem}[thm]{Lemma}

\theoremstyle{definition}
\newtheorem{rem}[thm]{Remark}

\newcommand{\bC}{{\mathbb{C}}}

\newcommand{\A}{{\mathcal{A}}}
\newcommand{\B}{{\mathcal{B}}}
\renewcommand{\H}{{\mathcal{H}}}
\newcommand{\K}{{\mathcal{K}}}
\newcommand{\M}{{\mathcal{M}}}
\renewcommand{\S}{{\mathcal{S}}}

\newcommand{\rC}{{\mathrm{C}}}

\newcommand{\ep}{\varepsilon}
\renewcommand{\phi}{\varphi}

\newcommand{\fA}{{\mathfrak{A}}}
\newcommand{\fB}{{\mathfrak{B}}}

\newcommand{\qand}{\quad\text{and}\quad}
\newcommand{\qfor}{\quad\text{for}\ }
\newcommand{\qforal}{\quad\text{for all}\ }
\newcommand{\AND}{\text{ and }}

\newcommand{\cconv}{\overline{\operatorname{conv}}}

\newcommand{\spn}{\operatorname{span}}

\newcommand{\CP}{\operatorname{CP}}
\newcommand{\UCP}{\operatorname{UCP}}

\newcommand{\ca}{\mathrm{C}^*}
\newcommand{\cenv}{\mathrm{C}^*_{\text{env}}}
\newcommand{\ol}{\overline}
\newcommand{\ucp}{u.c.p.\ }
\newcommand{\cp}{c.p.\ }
\newcommand{\cc}{c.c.\ }

\begin{document}
\title[The Choquet boundary]{The Choquet boundary of\\ an operator system}

\author[K.R. Davidson]{Kenneth R. Davidson}
\address{Department of Pure Mathematics\\University of Waterloo\\
Waterloo, ON\; N2L 3G1 \\ Canada}
\email{krdavids@math.uwaterloo.ca}

\author[M. Kennedy]{Matthew Kennedy}
\address{Department of Mathematics and Statistics\\ Carleton University\\
Ottawa, ON \; K1S 5B6 \\Canada}
\email{mkennedy@math.carleton.ca}

\begin{abstract}
We show that every operator system (and hence every unital operator algebra) 
has sufficiently many boundary representations to generate the C*-envelope.
\end{abstract}

\thanks{Both authors partially supported by research grants from NSERC (Canada).}
\maketitle

%%%%%%%%%%%%%%%%%%%%%%%%%%%%%%%%%%%%%%%%%%
We solve a 45 year old problem of William Arveson that is central to his approach
to non-commutative dilation theory. We show that every operator system and every
unital operator algebra has sufficiently many boundary representations to completely
norm it. Thus the C*-algebra generated by the image of the direct sum of these maps
is the C*-envelope. This was a central problem left open in Arveson's seminal work
\cite{Arv1969} on dilation theory for arbitrary operator algebras. 
In the intervening years, the existence of the C*-envelope was established,
but a general argument producing boundary representations has not been available.

Arveson  \cite{Arv1969, Arv1972} reformulated the classical dilation theory of Sz. Nagy \cite{SzNF}
so that it made sense for an arbitrary unital closed subalgebra $\A$ of a C*-algebra.
A central theme was the use of completely positive and completely bounded maps.
He proposed the existence of a family of special representations of $\A$, called
\textit{boundary representations}, which have unique
completely positive extensions to $\ca(\A)$ that are \textit{irreducible}
$*$-representations. The set of boundary representations is a noncommutative analogue of the
Choquet boundary of a function algebra, i.e. the set of points with unique representing measures.
Arveson proposed that there should be sufficiently many boundary representations, so that their
direct sum recovers the norm on $\M_n(\A)$ for all $n\ge1$. In this case, he showed that the
C*-algebra generated by this direct sum enjoys an important universal property, and provides 
a realization of the \textit{C*-envelope} of $\A$.

Arveson was not able to prove the existence of boundary representations in general,
although in various concrete cases they can be exhibited. Consequently, he was also unable
to prove the existence of the C*-envelope. However, a decade later, Hamana \cite{Ham}
established the existence of the C*-envelope using other methods. His proof, via the construction of a
minimal injective operator system containing $\A+\A^*$, did little to answer questions
about boundary representations. Nevertheless, it did lead to a variety of cases in which
the C*-envelope can be explicitly described. (We will not review the extensive literature
on this topic.) 

Nearly 20 years later, Muhly and Solel \cite{MuhlySolel} showed that boundary representations
(and more generally, $*$-representations that factor through the C*-envelope) have
homological properties that distinguish them from other representations. However, since their argument
relied on Hamana's theorem, it did not lead to a new construction of the C*-envelope.

About a decade ago, Dritschel and McCullough \cite{DritMcCull} came
up with an exciting new proof of the existence of the C*-envelope. It was a bona fide dilation
argument, building on ideas of Agler \cite{Agler}, and introduced the idea of
\textit{maximal dilations}. This direct dilation theory approach had the following important
consequence: if you begin with a completely isometric representation of $\A$, and find a maximal
dilation, then the C*-algebra generated by the image of this dilation is the C*-envelope. 
Consequently, there has been considerable interest in maximal dilations.

Arveson \cite{Arv2008} revisited the problem of the existence of boundary representations using
the ideas of Dritschel and McCullough. Using the disintegration theory of representations of C*-algebras,
he established that, in the separable case, sufficiently many boundary representations exist.
He expressed regret at the time that these delicate measure-theoretic 
methods appeared to be necessary---but reminded the audience  he had been looking for 
\textit{any} way of doing it for nearly 40 years\footnote{At the Fields Institute in Toronto, July, 2007,
in response to a question from Richard Kadison.}.

It is therefore of interest that our proof is a direct dilation-theoretic argument, building on ideas
from Arveson's original 1969 paper, and the more recent work of Dritschel and McCullough. 
In particular, our arguments do not require any disintegration theory nor do they require separability.

Arveson observed in his original work that a completely contractive unital map of
$\A$ into $\B(\H)$ extends uniquely to a self-adjoint map on the operator system
$\S=\ol{\A+\A^*}$ which is unital and completely positive. Consequently, he formulated much of
his theory around dilations of completely positive maps of operator systems.
We also work in this more general setting. 

Arveson developed many other important ideas in his seminal paper. 
One example which is particularly relevant to our work is the notion of a \textit{pure} completely positive map.
He showed that a completely positive map defined on a C*-algebra is pure if and only if
the minimal Stinespring dilation is irreducible. 
For completely positive maps on general operator systems, this is a necessary
but not sufficient condition.

We begin our approach by showing that every pure unital completely positive map on 
an operator system $\S$ has a pure maximal dilation. 
This dilation has a unique extension to $\ca(\S)$ which is an irreducible 
$*$-representation that necessarily factors through the C*-envelope. 
In other words, it is a boundary representation.

Then some results of Farenick \cite{Far00, Far04} are then used to show that  
there are sufficiently many finite dimensional pure \ucp maps (a.k.a \textit{matrix states}) to completely norm $\S$. 
Dilating these matrix states to boundary representations then yields a sufficient family of boundary representations.

Craig Kleski \cite{Kleski} has some closely related results. In the separable case, he uses
Arveson's measure theoretic approach to show that pure states have dilations to 
boundary representations. Also in connection with the second part of our paper, he shows that
the pure states on $\S$ norm it, and in the separable case, the supremum is attained.
He does not show that pure states completely norm $\S$, which we need.
In a private communication, Kleski showed us how his techniques yield a shorter
proof of  the second part of our argument. He has kindly allowed us to include it here.

%%%%%%%%%%%%%%%%%%%%%%%%%%%%%%%%%%%%%%%%%%
\section{Background}

We refer the reader to Paulsen's book \cite{Paulsen} for the background needed for this paper.
For a nice treatment of  maximal dilations (a l\`a Dritschel-McCullough), see
section 2 of \cite{Arv2008}. We briefly recall the central notions that we require.

An \textit{operator system} $\S$ is a unital norm-closed self-adjoint subspace of a C*-algebra.
We always view $\S$ as being contained in the C*-algebra that it generates, $\ca(\S)$.
Sometimes these are called concrete operator systems.
Choi and Effros \cite{ChoiEffros} gave an abstract axiomatic definition of an operator system, 
and established a representation theorem
showing that they can all be represented as concrete operator systems.

A \textit{unital operator algebra} $\A$ is a closed unital subalgebra of a C*-algebra.
Again, there is a definition of an abstract operator algebra, and
a corresponding representation theorem due to Blecher, Ruan and Sinclair \cite{BRS}
showing that they can all be represented (completely isometrically) as subalgebras
of C*-algebras. So our theory applies  to both abstract operator algebras and abstract operator systems.
For our purposes, we will assume that $\S$ or $\A$ is already sitting in a C*-algebra.

A map $\phi$ from any subspace $\M$ of a C*-algebra $\fA$ into a C*-algebra $\fB$
determines a family of maps $\phi_n:\M_n(\M) \to \M_n(\fB)$ given by 
$\phi_n([a_{ij}]) = [\phi(a_{ij})]$.
Say that $\phi$ is \textit{completely bounded} if 
\[ \|\phi\|_{cb} = \sup_{n\ge1} \|\phi_n\| < \infty .\]
Say that $\phi$ is \textit{completely contractive} (c.c.) if $\|\phi\|_{cb}\le1$.
If the domain of $\phi$ is an operator system $\S$, say that $\phi$ is \textit{completely positive} (c.p.)
if $\phi_n$ is positive for all $n\ge1$; and say that $\phi$ is \textit{unital completely positive} (u.c.p.)
if $\phi(1) = 1$. Since $\|\phi\|_{cb} = \|\phi(1)\|$ for \cp maps, we see that
\ucp maps are always completely contractive.

As mentioned in the introduction, every \textit{unital} completely contractive map $\phi$ of a 
\textit{unital} operator space $\M$ into a C*-algebra has a unique self-adjoint extension 
to $\S=\ol{\M+\M^*}$ given by 
\[ \tilde\phi(a+b^*) = \phi(a) + \phi(b)^* .\] Moreover, this map $\tilde\phi$ is completely positive.

A \ucp map $\phi:\S \to \B(\H)$ (or a \cc representation of an operator algebra $\A$) 
has the \textit{unique extension property} if it has a unique
\ucp extension to $\ca(\A)$ which is a $*$-representation.
If, in addition, the $*$-representation is irreducible, it is called a \textit{boundary representation}.
When $\A$ is a function algebra contained in $\rC(X)$, the irreducible $*$-representations
are just point evaluations. The restriction of a point evaluation to $\A$ has the unique extension property
if it has a  unique representing measure (namely, the point mass at the point itself).

A \textit{dilation} of a \cc unital representation $\rho : \A \to \B(\H)$ of an operator algebra $\A$ is a 
representation $\sigma:\A\to\B(\K)$ where $\K$ is a Hilbert space containing $\H$
such that $P_\H \sigma(a)|_\H = \rho(a)$ for $a \in \A$.
Similarly a \textit{dilation} of a \ucp map $\phi :\ S \to \B(\H)$ of an operator system $\S$ is a 
\ucp map $\psi : \S \to \B(\K)$ where $\K$ is a Hilbert space containing $\H$
such that $P_\H \psi(s)|_\H = \phi(s)$ for $s \in \S$.
We will write $\phi \prec \psi$ or $\psi \succ \phi$ to denote that $\psi$ dilates $\phi$.
The map ($\rho$ or $\phi$) is called \textit{maximal} if every dilation (of $\rho$ or $\phi$) is
obtained by attaching a direct summand (i.e.\ $\psi\succ \phi$ implies $\psi = \phi \oplus \psi'$ for some $\psi'$).

As noted above, a representation $\rho$ of an operator algebra $\A$ extends to a 
unique \ucp map $\tilde\rho$ on the operator system $\S = \ol{\A+\A^*}$. 
It is easy to see that a dilation $\sigma$ of $\rho$ extends to a dilation
$\tilde\sigma$ of $\tilde\rho$. However, this does not work in reverse. Indeed, a dilation of $\tilde\rho$
need not be multiplicative on $\A$, in which case it is not the extension of a representation.

Dritschel and McCullough \cite{DritMcCull} show that \cc representations of an operator algebra $\A$
always have maximal dilations. Arveson \cite{Arv2008} has a somewhat
nicer proof, along similar lines, which is valid for \ucp maps on an operator system $\S$.
Dritschel and McCullough show that maximal dilations extend to $*$-representations
of $\ca(\A)$. Arveson \cite{Arv2008} shows that being a maximal dilation of a \ucp map on $\S$
is equivalent to having the unique extension property. Thus a maximal dilation of a \ucp map is multiplicative. 
This implies that if $\rho$ is a \cc representation of $\A$, and $\psi$ is a maximal dilation of $\tilde\rho$,
then $\psi|_\A$ is a maximal dilation of $\rho$.
So establishing results for operator systems recovers the results for operator algebras at the same time.

The \textit{C*-envelope} of an operator system $\S$ consists of a C*-algebra $\fA=:\cenv(\S)$
and a completely isometric unital imbedding $\iota: \S \to \fA$ such that $\fA = \ca(\iota(\S))$, 
with the following universal property: whenever $j:\S\to\fB = \ca(j(\S))$ is a unital completely isometric
map, then there is a $*$-homomorphism $\pi:\fB\to\fA$ such that $\iota = \pi j$.
Hamana \cite{Ham} proved that the C*-envelope always exists.
Dritschel and McCullough \cite{DritMcCull} gave a new proof by showing that any maximal \ucp map on $\S$
extends to a $*$-representation of $\ca(\S)$ which factors through $\cenv(\S)$. 
In particular, when the original map is completely isometric, the maximal dilation yields a $*$-representation
onto the C*-envelope.

Arveson \cite{Arv1969} calls a \cp map $\phi$ \textit{pure} if the only \cp maps
satisfying $0 \le \psi \le \phi$ are scalar multiples of $\phi$. When $\phi$ is defined
on a C*-algebra $\fA$, it has a unique minimal Stinespring dilation $\phi(a) = V^*\pi(a)V$,
where $\pi$ is a $*$-representation of $\fA$ on $\K$ and $V \in \B(\H,\K)$.
Arveson shows that the intermediate \cp maps $\psi$ are precisely those maps 
of the form $\psi(a) = V^*T\pi(a)V$,
for $T \in \pi(\fA)'$ with $0 \le T \le I$.  Moreover, this is a bijective correspondence.
Thus, a \cp map on $\fA$ is pure if and only if the minimal Stinespring dilation is irreducible.
For a \cp map $\phi$ on an operator system $\S$, the minimal Stinespring dilation is not unique.
However, $\phi$ is not pure if any minimal Stinespring representation is reducible.

We will observe that if $\phi$ is maximal and pure, then it extends to an irreducible
$*$-representation of $\ca(\S)$. Our goal will be to establish that every pure \ucp map
from $\S$ into $\B(\H)$ has a pure maximal dilation which is a boundary representation. 
This will be accomplished in Section 2.
In Section 3, we gather the details needed to show that there are enough
boundary representations to completely norm $\S$, so that their direct sum provides
a completely isometric maximal representation of $\S$. 
This relies on results of Farenick \cite{Far00, Far04} on pure matrix states of
operator systems, based on the Krein-Milman type theorem for matrix convex sets
due to Webster and Winkler \cite{WebWink}. Altogether, our results establish that there
are sufficiently many boundary representations to construct the C*-envelope.

%%%%%%%%%%%%%%%%%%%%%%%%%%%%%%%%%%%%%%%%%%
\section{Extending pure maps}

First a simple observation mentioned in the preceding section.

\begin{lem} \label{L:boundary}
Every pure maximal \ucp map $\phi:\S \to \B(\H)$ extends to an irreducible
$*$-representation of $\ca(\S)$, and hence is a boundary representation.
\end{lem}

\begin{proof}
Arveson \cite{Arv2008} showed that maximal \ucp maps have the unique extension property.
So $\phi$ extends uniquely to a $*$-representation $\pi$ of $\ca(\S)$.
It remains to show that $\pi$ is irreducible. If $\pi$ is not irreducible,
then there is a proper projection $P$ commuting with $\pi(\ca(\S))$.
Thus $\psi(s) = P\phi(s)$ is a \cp map such that $0 \le \psi \le \phi$.
However, $\psi(1) = P$ is not a scalar multiple of $I = \phi(1)$.
So $\phi$ is not pure, contrary to our hypothesis. Hence $\pi$ is irreducible,
and therefore is a boundary representation.
\end{proof}

The proof in \cite{Arv2008} that maximal dilations exist uses the following concept.
A \ucp map $\phi$ is \textit{maximal at $(s_0,x_0)$} for $s_0\in\S$ and $x_0\in\H$
if whenever $\psi \succ \phi$, we have $\psi(s_0)x_0 = \phi(s_0)x_0$.
It is clear that this is true precisely when $\|\psi(s_0)x_0\| = \|\phi(s_0)x_0\|$
for all $\psi \succ\phi$.

The BW topology on $\B(\S,\B(\H))$ is the point-weak-$*$ topology.
An easy application of the Banach-Alaoglu Theorem shows that the unit
ball is compact since, in the BW topology, it embeds as a closed subset 
of the product of closed balls of $\B(\H)$ with the weak-$*$ topology.
In fact, $\B(\S,\B(\H))$ is a dual space, with the BW topology coinciding with the
weak-$*$ topology on bounded sets \cite[Lemma 7.1]{Paulsen}; but we do not need this fact.
The \cp and \ucp maps are closed in this topology \cite{Arv1969};
and thus the set of \ucp maps is BW-compact.

%%%%%%%%%%%%%%%%%%%%%%%%%%%%%%%%%%%%%%%%%%
\begin{lem} \label{L:one_step_extn}
Let $\S$ be an operator system, and let $\phi:\S \to \B(\H)$ be a \ucp map.
Given $s_0\in \S$ and $x_0\in\H$, there is a \ucp dilation of $\phi$ to a map 
$\psi: \S \to \B(\H\oplus\bC)$ which is maximal at $(s_0,x_0)$, i.e. 
\[ \|\psi(s_0)x_0\| = \sup \{ \|\rho(s_0)x_0\| : \rho \succ \phi \}. \]
\end{lem}

\begin{proof}
First note that if $\rho \succ \phi$, then the compression of $\rho$ to 
the Hilbert space $\spn\{\H, \rho(s)x\}$
yields a \ucp dilation $\rho'$ of $\phi$ into $\H \oplus \bC$ with $\|\rho'(s_0)x_0\|=\|\rho(s_0)x_0\|$.
So the supremum is the same if we consider only \ucp maps into $\B(\H\oplus\bC)$.
The set of all such maps is compact in the BW topology. Hence a routine compactness
argument yields the desired map $\psi$.
\end{proof}

This next lemma is motivated by Farenick's result \cite[Theorem B]{Far00} which 
states that a matrix state is pure if and only if it is a matrix extreme point.
However our arguments will work in Hilbert spaces of arbitrary dimension.
The goal is to construct a one dimensional dilation of a pure \ucp map to a \ucp map 
which is maximal at $(s_0,x_0)$ while conserving purity.

%%%%%%%%%%%%%%%%%%%%%%%%%%%%%%%%%%%%%%%%%%
\begin{lem} \label{L:extend_pure}
Let $\S$ be an operator system, and let $\phi:\S \to \B(\H)$ be a pure \ucp map.
Given $s_0\in \S$ and $x_0\in\H$ at which $\phi$ is not maximal, there is a pure \ucp dilation
$\psi: \S \to \B(\H\oplus\bC)$ which is maximal at  $(s_0,x_0)$.
\end{lem}

\begin{proof}
Let 
\[
 L = \sup \{ \|\rho(s_0)x_0\| : \rho \succ \phi \}
 \qand
 \eta = (L^2 - \|\phi(s_0)x_0\|^2 )^{1/2} .
\]
Let 
\[ X = \{ \psi \in \UCP(\S, \B(\H\oplus\bC)) : \psi \succ \phi \AND \psi(s_0)x_0 = \phi(s_0)x_0 \oplus \eta \} .\]
Conjugating the map obtained in the previous lemma by a unitary of the form $I \oplus \zeta$ implies that this is a non-empty BW-compact convex set.
Let $\psi_0$ be an extreme point of $X$.
Note that $X$ is a face of 
\[
 Y = \{ \psi \in \UCP(\S, \B(\H\oplus\bC)) : \psi \succ \phi \} .
\]
Hence $\psi_0$ is also an extreme point of $Y$.

We claim that $\psi_0$ is pure. To this end, suppose that $\psi_1$ is a \cp map
into $\H \oplus \bC$ such that $0 \le \psi_1 \le \psi_0$.
Set $\psi_2 = \psi_0 - \psi_1$.
To avoid the possibility that $\psi_i(1)$ may not be invertible, take a small $\ep>0$
and use 
\[
 \psi'_i = (1-2\ep) \psi_i + \ep \psi_0 \qfor i=1,2 . 
\]
Then $\psi_0 = \psi'_1 + \psi'_2$ and $\psi'_i(1) =: Q_i \ge \ep I$.
Thus $Q_i$ is invertible, and $Q_1+Q_2 = \psi_0(1) = I$.
If we show that $\psi_1'$ is a scalar multiple of $\psi_0$, then the same follows for $\psi_1$.

Observe that $P_\H \psi'_i(\cdot) |_\H \le \phi$.
By purity of $\phi$, there are positive scalars $\lambda_i$ 
so that $P_\H \psi'_i(\cdot) |_\H = \lambda_i \phi$.
Clearly $\lambda_1 + \lambda_2 = 1$ and $\lambda_i \ge \ep$.
Thus writing $Q_i$ as a matrix with respect to the decomposition $\H \oplus \bC$,
there is a vector $x_i \in \H$ and scalar $\alpha_i$ so that
\[
 Q_i = 
 \begin{bmatrix} \lambda_i & \lambda_i^{1/2} x_i\\
 \lambda_i^{1/2} x_i^* & \alpha_i \end{bmatrix} 
 = 
 \begin{bmatrix} \lambda_i ^{1/2} & 0\\
 x_i^* & \beta_i \end{bmatrix} 
 \begin{bmatrix} \lambda_i^{1/2} &x_i\\
 0 & \beta_i \end{bmatrix} .
\]
The factorization is possible by the Cholesky algorithm, 
where the positivity and invertibility of $Q_i$ guarantee that $\alpha_i > 0$ and
\begin{equation}
\beta_i = (\alpha_i - \|x_i\|^2)^{1/2} > 0. \label{eq:beta_i_alpha_i}
\end{equation}
Since $Q_1+Q_2=I$, we obtain
\begin{equation}
\lambda_1 + \lambda_2 = 1, \label{eq:lambda_i}
\end{equation}
\begin{equation}
\lambda_1^{1/2} x_1 + \lambda_2^{1/2} x_2= 0  \label{eq:lambda_i_x_i}
\end{equation}
and
\begin{equation}
\alpha_1 + \alpha_2 = 1. \label{eq:alpha_i}
\end{equation}
Let 
\[
 \gamma_i =  \begin{bmatrix} \lambda_i^{1/2} &x_i\\ 0 & \beta_i \end{bmatrix} ;
 \quad\text{then}\quad
 \gamma_i^{-1} =  \begin{bmatrix} \lambda_i^{-1/2} &-\lambda_i^{-1/2} \beta_i^{-1} x_i\\ 0 & \beta_i^{-1} \end{bmatrix} .
\] 
Also,
\[
 \gamma_1^*\gamma_1 + \gamma_2^*\gamma_2 = Q_1+Q_2 = I.
\]

Define \ucp maps 
\begin{align*}
 \tau_i(s) &= \gamma_i^{-1*} \psi_i'(s) \gamma_i^{-1}  \\&= 
 \begin{bmatrix} \lambda_i^{-1/2} & 0\\ * & * \end{bmatrix} 
 \begin{bmatrix} \lambda_i \phi(s) & *\\ * & * \end{bmatrix} 
 \begin{bmatrix} \lambda_i^{-1/2} & *\\ 0 & * \end{bmatrix}  \\&=
 \begin{bmatrix} \phi(s) & S_i(s)\\ T_i(s) & f_i(s) \end{bmatrix} . 
\end{align*}
Here $S_i \in \B(\S,\H)$, $T\in \B(\S,\H^*)$ and $f_i$ is a state on $\S$.
Since $\tau$ is positive, we have that $T_i(s) = S_i(s^*)^*$, but we will not require this.
Note that $\tau_i$ is a \ucp map such that
\[
P_\H \tau_i(\cdot) |_\H =  \phi.
\]
Hence $\tau_i$ is a dilation of $\phi$.  
Hence from the definition of $\eta$, we see that
\begin{equation}
  |T_i(s_0) x_0| \le \eta . \label{eq:T_i}
\end{equation}

Moreover,
\begin{align*}
 \psi_i'(s) &= \gamma_i^* \tau_i(s) \gamma_i  \\&=
 \begin{bmatrix} \lambda_i ^{1/2} & 0\\ x_i^* & \beta_i \end{bmatrix} 
 \begin{bmatrix} \phi(s) & S_i(s)\\ T_i(s) & f_i(s) \end{bmatrix} 
 \begin{bmatrix} \lambda_i ^{1/2} & x_i \\ 0 & \beta_i \end{bmatrix}  \\&=
 \begin{bmatrix} \lambda_i\phi(s) & * \\ 
 \lambda_i^{1/2} \big( x_i^* \phi(s) + \beta_i T_i(s) \big)  & * \end{bmatrix}  \! .
\end{align*}
Consideration of the lower left entries of $\psi_1'(s_0)$ and $\psi_2'(s_0)$ evaluated at $x_0$, combined with (\ref{eq:lambda_i_x_i}) yields 
\begin{align*}
 \eta &= P_{\bC} \psi_0(s_0) x_0 \\&=
 P_{\bC} \psi'_1(s_0)  x_0 +
 P_{\bC} \psi'_2(s_0)  x_0  \\ &=
 \lambda_1^{1/2} \big( x_1^* \phi(s_0) + \beta_1 T_1(s_0) \big)x_0 + 
 \lambda_2^{1/2} \big( x_2^* \phi(s_0) + \beta_2 T_2(s_0) \big)x_0  \\ &=
 (\lambda_1^{1/2} x_1 + \lambda_2^{1/2} x_2)^* \phi(s_0) x_0 + 
 \lambda_1^{1/2}\beta_1 T_1(s_0) x_0 + \lambda_2^{1/2}\beta_2 T_2(s_0) x_0  \\ &=
 \lambda_1^{1/2}\beta_1 T_1(s_0) x_0 + \lambda_2^{1/2}\beta_2 T_2(s_0) x_0 .  
\end{align*}
Therefore, 
\begin{align*}
 \eta &\le (\lambda_1^{1/2} \beta_1 + \lambda_2^{1/2} \beta_2)\eta \\
 &\le (\lambda_1^{1/2} \alpha_1^{1/2} + \lambda_2^{1/2} \alpha_2^{1/2})\eta \\
 & \le (\lambda_1 + \lambda_2)^{1/2} (\alpha_1 + \alpha_2)^{1/2} \eta \\
 & = \eta,
\end{align*}
where we have used (\ref{eq:T_i}), (\ref{eq:beta_i_alpha_i}), (\ref{eq:lambda_i}), (\ref{eq:alpha_i}) and the Cauchy-Schwarz inequality.

Since this is an equality, the last inequality yields that $\beta_i^2 = \alpha_i$, and hence $x_i=0$. Furthermore, 
 an equality in the use of the Cauchy-Schwarz inequality implies that the unit vectors 
 \[(\lambda_1^{1/2},\lambda_2^{1/2}) \quad \mathrm{and} \quad (\alpha_1^{1/2},\alpha_2^{1/2})\]
  are collinear, and hence must be equal. Thus $\beta_i = \alpha_i^{1/2} = \lambda_i^{1/2}$ and $\gamma_i = \lambda_i^{1/2}I$ is a scalar matrix.

{}From above,
\[
\psi_i'(\cdot) = \gamma_i^* \tau_i(\cdot) \gamma_i = \lambda_i\tau_i(\cdot).
\]
Therefore $\psi_0 = \lambda_1 \tau_1 + \lambda_2 \tau_2$ is a convex combination of the $\tau_i$.
Since $\psi_0$ is an extreme point of $X$, we obtain that $\tau_i = \psi_0$.
Thus $\psi'_1 = \lambda_1 \psi_0$. So $\psi_0$ is pure.
\end{proof}

It is easy to see that $\phi$ is maximal if and only if it is maximal at every $(s,x)$
for $s\in\S$ and $x\in\H$ \cite{Arv2008}. We will establish the existence of pure maximal dilations by a transfinite induction.
In the separable case, a simple induction is possible.

\pagebreak[3]
%%%%%%%%%%%%%%%%%%%%%%%%%%%%%%%%%%%%%%%%%%
\begin{thm} \label{T:puremax}
Let $\S$ be an operator system, and let $\phi:\S \to \B(\H)$ be a pure \ucp map.
Then $\phi$ has a pure maximal dilation $\psi$. 
Therefore $\psi$ extends to a $*$-representation of $\ca(\S)$ 
which is a boundary representation of $\S$.
\end{thm}

\begin{proof}
Enumerate a dense subset of  $b_1(\S) \times b_1(\H)$, the product of unit balls of $\S$ and $\H$, 
using an ordinal $\Lambda$, as 
\[
 \{ (s_\lambda, x_\lambda) : \lambda < \Lambda, \text{ such that } \lambda \text{ is a successor ordinal} \} . 
\]
We will use transfinite induction to construct a pure \ucp dilation of $\phi$
which is maximal at each $(s_\lambda, x_\lambda)$.

Start with $\phi_0 := \phi$. At each successor ordinal $1+\lambda$ for $\lambda\ge 0$,
we have a pure \ucp dilation $\phi_\lambda$ of $\phi$ into a Hilbert space $\H_\lambda$
which is maximal at $(s_\alpha, x_\alpha)$ for all $\alpha \le \lambda$.
If $\phi_\lambda$ is already maximal at $(s_{1+\lambda}, x_{1+\lambda})$, 
set $\phi_{1+\lambda}=\phi_\lambda$.
Otherwise, use  Lemma~\ref{L:extend_pure} to obtain a 1-dimensional dilation of $\phi_\lambda$ to a pure \ucp map
$\phi_{1+\lambda}$ into a Hilbert space $\H_{\lambda + 1}$ which is maximal at $(s_{1+\lambda}, x_{1+\lambda})$.

At each limit ordinal $\mu$, for each $\alpha < \mu$, we have a Hilbert space
$\H_\alpha$ and pure \ucp dilation $\phi_\alpha$ which is maximal at $(s_\lambda, x_\lambda)$
for each successor ordinal $\lambda \le \alpha$. Moreover if $\lambda < \alpha$,
then $\H_\lambda \subset \H_\alpha$ and $\phi_\lambda \prec \phi_\alpha$.
Let $\H_\mu$ be the direct limit of the Hilbert spaces $\H_\alpha$,
which we can consider as the completion of the union $\bigcup_{\alpha < \mu}\H_\alpha$.
Then we define $\phi_\mu$ so that the compression of $\phi_\mu$ to $\H_\alpha$ is $\phi_\alpha$
for all $\alpha < \mu$. Clearly $\phi_\mu$ is a \ucp map which is a dilation of $\phi_\alpha$
for each $\alpha < \mu$. To see that $\phi_\mu$ is pure, suppose that $0 \le \tau \le \phi_\mu$.
The compression of $\tau$ to $\H_\alpha$ satisfies 
\[ 0 \le P_{\H_\alpha} \tau(\cdot) |_{\H_\alpha} \le \phi_\alpha .\]
By purity, there is a scalar $t$ so that $P_{\H_\alpha} \tau(\cdot) |_{\H_\alpha} = t \phi_\alpha$.
Moreover $tI = P_{\H_\alpha} \tau(1) |_{\H_\alpha}$; so $t$ is independent of $\alpha$.
By continuity, $\tau = t \phi_\mu$ and hence $\phi_\mu$ is pure.

The result at the end of this induction is a pure \ucp dilation $\psi_1$ of $\phi$ 
acting on a Hilbert space $\K_1$, which by continuity is maximal at $(s,x)$ for every $s\in\S$ and $x \in \H$.
Repeat this procedure recursively to obtain a sequence of pure \ucp dilations $\psi_k$ acting on  $\K_k$
which are maximal at $(s,x)$ for every $s\in\S$ and $x \in \K_{k-1}$.
The direct limit of this sequence is a pure \ucp dilation $\psi_\infty$ acting on $\K_\infty$
which is maximal at $(s,x)$ for every $s\in\S$ and $x \in \K_\infty$;
and thus is maximal.
Arguing as in the limit ordinal case above, $\psi_\infty$ is pure.
Finally, by Lemma~\ref{L:boundary}, $\psi_\infty$ extends to an irreducible 
$*$-representation of $\ca(\S)$ which is a boundary representation of $\S$.
\end{proof}

%%%%%%%%%%%%%%%%%%%%%%%%%%%%%%%%%%%%%%%%%%
\begin{rem}
If $\H$ is finite dimensional and $\S$ is separable, the intermediate dilations of the previous proof can be kept finite dimensional,
so that only the final limit dilation is infinite dimensional. This is accomplished by doing the
dilations at the $k$-th stage only for a finite set  of pairs $(s_i,x^k_j)$,  where 
$1 \le i \le N_k$ and $\{x^k_j\}$ forms a finite $\ep_k$-net in the unit sphere of $\H_k$. Here, 
$N_k$ and $\ep_k$ are chosen such that  $\lim_k N_k = \infty$ and $\lim_{k\to\infty} \ep_k = 0$.
In the limit, one still obtains a maximal dilation.
\end{rem}

%%%%%%%%%%%%%%%%%%%%%%%%%%%%%%%%%%%%%%%%%%%
\section{Sufficiency of boundary representations}

Now that we have a method for constructing boundary representations,
we show that there are enough of them to yield the C*-envelope.
We provide two arguments. The first very slick argument is due to
Craig Kleski \cite{Kleski_email}, and we thank him for allowing us to include this proof here.
This argument is based on states on $\M_n(\S)$.

\begin{thm} \label{T:norm}
If $S\in \M_n(\S)$, then there is a boundary representation $\pi$ of $\S$
such that $\|S\| = \|\pi^{(n)}(S)\|$.
\end{thm}

\begin{proof}
It suffices to accomplish this for $T=S^*S$, since then
\[ \|\pi^{(n)}(S)\|^2 = \|\pi^{(n)}(T)\| = \|T\| = \|S\|^2 .\]
It is a standard argument that there is a state $f$ on $\ca(T)$
such that $f(T) = \|T\|$. Extend this by the Hahn-Banach Theorem to
$\M_n(\ca(\S))$ to obtain a state that norms $T$. The set
\[ \{ f \in S(\M_n(\ca(\S))) : f(T) = \|T\| \} \]
is a weak-$*$ compact convex set.  By the Krein-Milman Theorem,
it has an extreme point $f_0$. This is a pure state of $\M_n(\ca(\S))$.

By Theorem~\ref{T:puremax}, there is a boundary representation $\sigma$
of $\M_n(\S)$ that dilates $f_0$. Clearly $\|\sigma(T)\|=\|T\|$.
Another standard argument shows that the representation $\pi$ of $\ca(\S)$ 
obtained by compression to the range of a matrix unit 
satisfies $\sigma \simeq \pi^{(n)}$.
Now the easy direction of Hopenwasser's theorem \cite{Hop} yields that $\pi$
is a boundary representation of $\S$. Indeed, if $\pi|_\S$ has a  
\ucp dilation $\phi$, then $\phi^{(n)}$ is a \ucp dilation of $\sigma|_{\M_n(\S)}$.
Since $\sigma$ has the unique extension property, $\phi^{(n)}=\pi^{(n)}$; and thus $\phi=\pi$.
Therefore $\pi$ is a boundary representation with the desired norming property.
\end{proof}

Now we turn to a second approach based on some interesting ideas of Farenick \cite{Far04}.
This argument is based on matrix states on $\S$ itself.
A \textit{matrix state} is a \ucp map of $\S$ into the $k\times k$ matrices $\M_k$.
Let $S_k(\S) = \UCP(\S,\M_k)$ be the set of all \ucp maps from $\S$ into $\M_k$.
The set of all matrix states is $S(\S) = ( S_k(\S) )_{k\ge1}$.

There is a natural bijective correspondence between $\CP(\S,\M_n)$ and $\CP(\M_n(\S),\bC)$; 
see \cite[Theorem~6.1]{Paulsen}. However this does not yield a correspondence between 
matrix states on $\S$ and states on $\M_n(\S)$. So we do not know a way to deduce
the existence of sufficiently many pure matrix states from the previous argument.
So this second approach is of independent interest.

We begin with an easy observation.

%%%%%%%%%%%%%%%%%%%%%%%%%%%%%%%%%%%%%%%%%%
\begin{lem} \label{L:norming}
The set of all matrix states completely norms $\S$; i.e.\
for every $S\in \M_n(\S)$,
\[
 \|S\| = \sup\{ \|\phi_n(S)\| : \phi \in S(\S) \} .
\]
\end{lem}

\begin{proof}
Let $\pi$ be a faithful $*$-representation of $\ca(\S)$ on $\H$.
Then $\pi$ is completely isometric. Hence the set of compressions of
$\pi$ to all finite dimensional subspaces of $\H$ also completely norms $\S$.
\end{proof}

Now the issue is to replace the set of all matrix states with the set of pure matrix states.
For this, we need the notions of matrix convexity and matrix extreme points.

A \textit{matrix convex set} in a vector space $V$ is a collection $K = (K_k)$
of subsets $K_k \subset \M_k(V )$ such that $K_k$ contains all elements of the form
\[
 \sum_{i=1}^p \gamma_i^* v_i \gamma_i  
 \qforal v_i \in K_{k_i} ,\ \gamma_i \in \M_{k_i,k}, \text{ such that } \sum_{i=1}^p \gamma_i^* \gamma_i = I_k .
\]
If $S=(S_k)$ is a collection of subsets of $\M_k(V)$, then there is a smallest
closed matrix convex set generated by $S$ called $\cconv(S)$.

A matrix convex combination $v=\sum_{i=1}^p \gamma_i^* v_i \gamma_i$ is \textit{proper} if
each $\gamma_i$ has a right inverse belonging to $\M_{k,k_i}$ , i.e., if $\gamma_i$ is surjective. 
In particular, we must have that $k\ge k_i$.
A point $v \in K_k$ is a \textit{matrix extreme point} if whenever $v$ is a proper matrix convex 
combination of $v_i\in K_{k_i}$ for $1\le i \le p$, then each $k_i=k$ and $v_i=u_ivu_i^*$ for some unitary $u_i\in\M_k$.
In particular, at level $k=1$, matrix extreme points are just extreme points.
Webster and Winkler \cite[Theorem 4.3]{WebWink} prove a Krein-Milman Theorem for 
matrix convex sets stating that a compact matrix convex set is the closed matrix convex hull
of its matrix extreme points.

The matrix state space $S(\S) = (S_k(\S))_{k\ge1}$ of an operator system $\S$ forms a 
BW-compact matrix convex set.
A result of Farenick \cite[Theorem B]{Far00} shows that a matrix state is pure if and only if it 
is a matrix extreme point of $S(\S)$.
Thus every matrix state is in the BW-closure of the matrix convex combinations of the
pure matrix states.

In \cite{Far04}, Farenick provides a simpler proof which is independent of these results.
His argument starts with the observation that the extreme rays of the \cp maps of $\S$ into $\M_n$
are precisely the pure \cp maps. Then he uses Choquet's Theorem on convex cones
to show that the convex hull of the pure \cp maps is BW-dense in the whole cone.
Now a normalization argument shows directly that the C*-convex combinations
of pure matrix states are BW-dense in the set of all matrix states.

%%%%%%%%%%%%%%%%%%%%%%%%%%%%%%%%%%%%%%%%%%
\begin{lem} \label{L:purenorming}
The set of all pure matrix states completely norms $\S$; i.e.\
for every $S\in \M_n(\S)$,
\[
 \|S\| = \sup\{ \|\phi_n(S)\| : \phi \in S(\S), \ \phi \text{ pure} \} .
\]
\end{lem}

\begin{proof}
It suffices to show that the supremum over all matrix convex combinations
of pure matrix states is no larger than the supremum over pure matrix states.
This inequality will then extend to the BW-closure by continuity.
Thus by the remarks preceding the lemma, this will be the supremum over
all matrix states. Hence the result follows from Lemma~\ref{L:norming}.

Suppose that $\phi\in S_k(\S)$ is a matrix convex combination of pure states $\phi_i \in S_{k_i}(\S)$.
So there are linear maps $\gamma_i\in \M_{k_i,k}$ such that
\[
 \phi = \sum_{i=1}^p \gamma_i^* \phi_i \gamma_i  
 \qand
 \sum_{i=1}^p \gamma_i^* \gamma_i = I_k .
\]
Then $\psi := \phi_1 \oplus \dots \oplus \phi_p$ belongs to $S_K(\S)$ where $K=\sum_{i=1}^p k_i$.
We can factor $\phi$ as
\[
 \phi = 
 \begin{bmatrix}\gamma_1\\ \gamma_2\\ \vdots\\ \gamma_p\end{bmatrix}^*
 \begin{bmatrix} 
 \phi_1 & 0 & \dots & 0\\
 0 & \phi_2 & \ddots & 0\\
 \vdots &\ddots&\ddots&\vdots\\
 0&\dots&0&\phi_p
 \end{bmatrix}
 \begin{bmatrix}\gamma_1\\ \gamma_2\\ \vdots\\ \gamma_p\end{bmatrix}
 = \gamma^* \psi \gamma,
\]
where $\gamma = (\gamma_1, \gamma_2,\ldots,\gamma_p)^T$.
Observe that 
\[ \gamma^*\gamma = \sum_{i=1}^p \gamma_i^* \gamma_i = I_k .\]
Hence $\gamma$ is an isometry.

Let $S \in \M_n(\S)$. Then
\begin{align*}
 \| \phi_n(S)\|  &= \| (\gamma \otimes I_n)^* \psi_n(S) (\gamma\otimes I_n) \| \\
 &\le \left\| 
 \begin{bmatrix} 
 (\phi_1)_n(S) & 0 & \dots & 0\\
 0 & (\phi_2)_n(S) & \ddots & 0\\
 \vdots &\ddots&\ddots&\vdots\\
 0&\dots&0&(\phi_p)_n(S)
 \end{bmatrix}
 \right\| \\
 &= \max_{1 \le i \le p} \| (\phi_i)_n(S) \| .
\end{align*}
The right hand side is a maximum over pure states, as desired.
\end{proof}

We can now combine all of the ingredients to obtain the main result.
%%%%%%%%%%%%%%%%%%%%%%%%%%%%%%%%%%%%%%%%%%%
\begin{thm} \label{T:sufficient}
Let $\S$ be an operator system. Then $\S$ is completely normed by its boundary representations.
Hence the direct sum of all boundary representations yields a completely isometric map 
$\iota:\S \to \B(\K)$, so that $(\iota,\ca(\iota(\S)))$ is the C*-envelope of $\S$. 
$($Here, the direct sum is taken over a set of fixed Hilbert spaces of dimensions
ranging from $1$ up to $\aleph_0 \dim \S .)$
\end{thm}

\begin{proof}
By Lemma~\ref{L:purenorming}, the pure matrix states completely norm $\S$.
By Theorem~\ref{T:puremax}, each of these pure matrix states can be dilated 
to a boundary representation of $\S$. Clearly this implies that the collection of 
all boundary representations completely norms $\S$.
To get a set, we need to take the precaution to fix a set of Hilbert spaces
of the proper dimensions to accomodate irreducible representations of $\ca(\S)$.
This dimension is bounded above by $\aleph_0 \dim \S$.
The direct sum $\pi$ of this set of boundary representations is then completely isometric on $\S$.
Each boundary representation is maximal, and thus any dilation of $\pi$ must leave 
each boundary representation as a direct summand.  
Hence $\pi$ is a direct summand of its dilation, and therefore is a maximal \ucp map. 
By the arguments of Dritschel and McCullough \cite{DritMcCull} or Arveson \cite{Arv2008},
the C*-envelope of $\S$ is the C*-algebra generated by this representation.
\end{proof}

Earlier remarks yield the corresponding result for operator algebras.

%%%%%%%%%%%%%%%%%%%%%%%%%%%%%%%%%%%%%%%%%%%
\begin{cor} \label{C:sufficient}
Let $\A$ be a unital operator algebra. Then $\A$ is completely normed by its boundary representations.
Hence the direct sum of all boundary representations yields a completely isometric map $\iota:\A \to \B(\K)$, so that
$(\iota,\ca(\iota(\A)))$ is the C*-envelope of $\A$.  
$($Here, the direct sum is taken over a set of fixed Hilbert spaces of dimensions
ranging from $1$ up to $\aleph_0 \dim \S .)$
\end{cor}

%%%%%%%%%%%%%%%%%%%%%%%%%%%%%%%%%%%%%%%%%%%


\begin{thebibliography}{99}

\bibitem{Agler} J. Agler,
\textit{An abstract approach to model theory},
in Surveys of some recent results in operator theory, Vol. II, pp.\ 1--23, 
Longman Sci. Tech., Harlow, 1988.

\bibitem{Arv1969} W. Arveson,
\textit{Subalgebras of C*-algebras},
Acta Math.\ \textbf{123} (1969), 141--224.

\bibitem{Arv1972} W. Arveson, 
\textit{Subalgebras of C*-algebras II}, 
Acta Math.\ \textbf{128} (1972), 271--308.

\bibitem{Arv2008}W. Arveson,
\textit{The noncommutative Choquet boundary},
J. Amer.\ Math.\ Soc. \textbf{21} (2008), 1065--1084.

\bibitem{BRS} D. Blecher, Z.J. Ruan and A. Sinclair,
\textit{A characterization of operator algebras},
J. Funct.\ Anal.\ \textbf{89} (1990), 188--201.

\bibitem{ChoiEffros} M.D. Choi and E. Effros,
\textit{Injectivity and operator spaces},
J. Funct.\ Anal.\ \textbf{24} (1977), 156--209.

\bibitem{DritMcCull} M. Dritschel and S. McCullough,
\textit{Boundary representations for families of representations of operator algebras and spaces},
J. Operator Theory \textbf{53} (2005), 159--167.

\bibitem{Far00} D. Farenick,
\textit{Extremal matrix states on operator systems},
J. London Math.\ Soc. \textbf{61}(3) (2000), 885--892. 

\bibitem{Far04} D. Farenick, 
\textit{Pure matrix states on operator systems},
Linear Algebra Appl. \textbf{393} (2004), 149--173. 

\bibitem{Ham} M. Hamana,
\textit{Injective envelopes of operator systems},
Publ.\ Res.\ Inst.\ Math.\ Sci.\ \textbf{15} (1979), 773-785.

\bibitem{Hop}  A. Hopenwasser, 
\textit{Boundary representations on C?-algebras with matrix units},
Trans.\ Amer.\ Math.\ Soc. \textbf{177} (1973), 483--490.

\bibitem{Kleski} C. Kleski,
\textit{Boundary representations and pure completely positive maps},
J. Operator Theory, to appear.

\bibitem{Kleski_email} C. Kleski, private communication, May 8, 2013.

\bibitem{MuhlySolel} P. Muhly and B. Solel,
\textit{An algebraic characterization of boundary representations},
Nonselfadjoint operator algebras, operator theory, and related topics, 189--196, 
Oper.\ Theory Adv.\ Appl.\ \textbf{104}, 
Birkhauser, Basel, 1998. 

\bibitem{Paulsen} V. Paulsen,
\textit{Completely bounded maps and operator algebras}, 
Cambridge Studies in Advanced Mathematics \textbf{78}, 
Cambridge University Press, Cambridge, 2002.

\bibitem{SzNF} B. Sz.~Nagy, C. Foia\c{s}, H. Bercovici and L. Kerchy, 
{\it Harmonic analysis of operators on Hilbert space}, 2nd ed.,
Springer Verlag, New York, 2010.

\bibitem{WebWink} C. Webster and S. Winkler,
\textit{The Krein-Milman theorem in operator convexity},
Trans.\ Amer.\ Math.\ Soc. \textbf{351} (1999), 307--322.

\end{thebibliography}
\end{document}